\documentclass{elsarticle}

\usepackage{amsmath}
\usepackage{amsfonts}

\newtheorem{proposition}{Proposition}
\newtheorem{corollary}{Corollary}
\newtheorem{example}{Example}

\newtheorem{lemma}{Lemma}
\newtheorem{theorem}{Theorem}
\newproof{proof}{Proof}
\newtheorem{definition}{Definition}

\newcommand{\RR}{\mathbb R}
\newcommand{\PP}{\mathbb P}

\newcommand{\mcV}{\mathcal{V}}

\newcommand{\ybdf}{y_k^\mathrm{BDF}}

\newcommand{\rank}{\mathop{\text{rank}}\nolimits}
\newcommand{\argmin}{\mathop{\text{argmin}}}
\newcommand{\diag}{\mathop{\text{diag}}\nolimits}
\newcommand{\range}{\mathop{\text{range}}\nolimits}

\begin{document}

\begin{frontmatter}

    \title{Minimal residual multistep methods for large stiff non-autonomous linear problems}

    \author[bsu]{B.V.~Faleichik\corref{cor1}}
    \ead{faleichik@bsu.by}


    \address[bsu]{Department of Computational Mathematics, Belarusian State University, 4,~Nezavisimosti~avenue, 220030, Minsk, Belarus}

    \cortext[cor1]{Corresponding author}


%

    \begin{abstract}
    The purpose of this work is to introduce a new idea of how to avoid the factorization of large matrices during the solution of stiff systems of ODEs. Starting from the general form of an explicit linear multistep method we suggest to adaptively choose its coefficients on each integration step in order to minimize the norm of the residual of an implicit BDF formula. Thereby we reduce the number of unknowns on each step from $n$ to $O(1)$, where $n$ is the dimension of the ODE system. We call this type of methods Minimal Residual Multistep (MRMS) methods. 
    In the case of linear non-autonomous problem, besides the evaluations of the right-hand side of ODE, the resulting numerical scheme additionally requires one solution of a linear least-squares problem with a thin matrix per step. We show that the order of the method and its zero-stability properties coincide with those of the used underlying BDF formula. 

    For the simplest analog of the implicit Euler method the properties of linear stability are investigated. Though the classical absolute stability analysis is not fully relevant to the MRMS methods, it is shown that this one-step method is applicable in stiff case. In the numerical experiment section we consider the fixed-step integration of a two-dimensional non-autonomous heat equation using the MRMS methods and their classical BDF counterparts. The starting values are taken from a preset slowly-varying exact solution. The comparison showed that both methods give similar numerical solutions, but in the case of large systems the MRMS methods are faster, and their advantage considerably increases with the growth of dimension. Python code with the experimantal code can be downloaded from the GitHub repository \textsf{https://github.com/bfaleichik/mrms}.
\end{abstract}

\end{frontmatter}

%
%
%
%
%
%
%

\section*{Abbreviations}

ODE --- ordinary differential equation

MRMS method --- minimal residual multistep method

\section*{Introduction}

The necessity of solving large systems of linear equations during the nummerical solution of stiff initial value problems of high dimensions is a well-known problem of implicit numerical methods. Currently there are several approaches which deal with this problem. `Purely explicit' methods for stiff problems employ only the evaluations of the right-hand side of an ODE, for example the well-known explicit Runge-Kutta methods of Chebyshev type \cite{VDH}, \cite{Lebedev}. There are also methods which implement implicit methods using special matrix-free iterative processes instead of conventional Newton iteration \cite{gpi}. The second kind are the methods which admit a certain level of `implicitness', i.~e. some linear systems (possibly of low dimensions) need to be solved. The most remarkable method of this kind is the so-called Newton-Krylov method for solution of nonlinear systems proposed by Brown and Hindmarsh \cite{Brown}. 

The idea which is developed in the present work lies in between the above mentioned tactics. In a certain sense we borrowed 
 the idea to restrict the set of possible solutions by a subspace of low dimension from \cite{Brown}. This allowed us to reduce the number of unknowns on each step from $n$ to $O(1)$, where $n$ is the dimension of the ODE system.
On the other hand there is some analogy to Rosenbrock methods \cite{Rosenbrock} which do not employ any iterative processes. In contrast to these methods, instead of one linear system per step we suggest to solve a linear least-squares problem with a thin matrix.

The paper is organized as follows. In section 
\ref{sec:spec} we define the general form of the method. In sections
\ref{sec:order} and \ref{sec:zero} the properties of accuracy and zero-stability are discussed. Section \ref{sec:linear-stab} is devoted to the investigation of linear stability properties of the MRMS methods. The main focus here is on the simplest method of the MRMS family --- the one-step minimal residual counterpart of the implicit Euler method. For the general case we perform a numerical investigation of linear stability in subsection \ref{ss:gen-lin-stab}. In section \ref{sec:implementation} we discuss the details of implementation of the methods for linear systems, and section \ref{sec:experiment} contains the results of the numerical experiment with a 2D heat equation which demonstrate considerable advantage of the MRMS methods over their BDF prototypes.

\section{Specification of the method}
\label{sec:spec}

Consider a system of ordinary differential equations
\begin{equation}\label{eq:ivp}
  y'=f(t,y),\quad y(t_0) = y_0\in\RR^n,
\end{equation}
and a $k$-step explicit linear multistep method of general form 
\begin{equation}\label{eq:yk}
y_k=\sum_{j=0}^{k-1} (\tau \beta_j f_j-\alpha_j y_j)
\end{equation}
with starting points 
$(t_j,y_j)$, $t_j<t_{j+1}$, $t_k=t_{k-1}+\tau$, $f_j=f(t_j,y_j)$. In contrast to the classical case the coefficients $\alpha=[\alpha_0,\ldots,\alpha_{k-1}]^T, \beta=[\beta_0,\ldots,\beta_{k-1}]^T$ on each step are not supposed to be determined explicitly by order conditions, but rather are subjected to some optimization constraints. More precisely, we would like to choose such coefficients that give a best approximation to $y(t_k)$ in the following sense. Consider a standard backward approximation of order $p$ to derivative
\begin{equation}
  \tau y'(t_k)\approx c_k y(t_k) + c_{k-1} y(t_{k-1})+ \ldots + c_{k-p} y(t_{k-p}),\quad p\leq k, 
\end{equation}   
a corresponding standard $p$-step BDF formula \cite{Curtiss}
\begin{equation}\label{eq:bdf}
    c_k y_k + c_{k-1} y_{k-1}\ldots + c_{k-p} y_{k-p} = \tau f_k,
\end{equation}   
and a residual function $r:\RR^n\to\RR^n$,
\begin{equation}\label{eq:r}
    r(x) = \tau f(t_k,x)-(c_k x + c_{k-1} y_{k-1}+\ldots + c_{k-p} y_{k-p}).
\end{equation}   
Then
\begin{equation}
    (\alpha,\beta) = \argmin_{\alpha',\beta'}\left\|r\Bigl(\sum_{j=0}^{k-1} (\tau \beta'_j f_j-\alpha'_j y_j)\Bigr)\right\|.
\end{equation}
Generally any norm in $\RR^n$ can be used, but for the practical use the 2-norm is more preferable.
For obvious reasons sometimes we will call \eqref{eq:bdf} an \emph{underlying BDF method}.
By $\ybdf$ we will denote the exact result of this method, i.~e. a vector for which $r(\ybdf)=0$ must hold. Generally speaking, $\ybdf$ does not necessarily need to be unique and may not exist at all (see e.~g. Example \ref{ex:ie-fail} below). 

If we rewrite \eqref{eq:yk} in the form
\begin{equation}
    y_k=V \gamma,   
\end{equation}
where $\gamma=\begin{bmatrix}\alpha\\ \beta\end{bmatrix}\in \RR^{2k}$, $V:\RR^{2k}\to \RR^n$,
\begin{equation}\label{eq:Vk}
    V=\Bigl[-y_0\Bigl|-y_1\Bigl|\ldots\Bigr|-y_{k-1}\Bigl| 
    \tau f_0\Bigl|\tau f_1\Bigl|\ldots\Bigr|\tau f_{k-1}\Bigr],      
\end{equation}
then an alternative point of view on the proposed method is seen. We are looking for an approximate solution of BDF equation \eqref{eq:bdf} in the subspace $\mcV\subset \RR^n$ spanned by vectors $y_0,\ldots y_{k-1}, \tau f_0,\ldots \tau f_{k-1}$:
\begin{equation}\label{eq:rmin}
    y_k=\argmin_{x\in \mcV}\|r(x)\|,\quad \mcV=\range V.     
\end{equation}
There is a simple but important observation: if $\ybdf\in \mcV$ and is unique, then $y_k=\ybdf$. In particular this may be the case when $\rank V=n$ and it easier to directly solve ${r(y_k)=0}$ than the optimization problem \eqref{eq:rmin}. Therefore the most interesting and important case is $2k\ll n$, so intrinsically MRMS methods are reasonable to apply only to big ODE systems.  Before proceeding to the main properties of the method we pin down the definition.
\begin{definition}
Having $k$ starting points $(t_j,y_j)$, $j=0,1,\ldots, k-1$, and an underlying $p$-step BDF formula \eqref{eq:bdf}, a minimal residual multistep method MRMS($k$,$p$) is defined by \eqref{eq:rmin}, where the residual function $r$ and the matrix $V$ are defined correspondingly by \eqref{eq:r} and \eqref{eq:Vk}.
\end{definition}

\section{Order of the method}
\label{sec:order}

Consider the fixed step case: $t_j=j \tau$.
\begin{theorem}\label{thm:mrms-order}
The order of a MRMS($k$, $p$) method is $\min\{2k-1, p\}$.
\end{theorem}
\begin{proof}
    The basic fact for the proof is that there always exists an explicit $k$-step linear multistep method of order $2k-1$. Denote the corresponding vector of coefficients by $\hat\gamma=[\hat \alpha,\hat \beta]^T$. Then for $\hat y_k=V \hat \gamma_k$ we have
    \[
        \hat y_k - y(t_k) = O(\tau^{2k}).
    \]
    By definition of $y_k$ for the MRMS method 
    \[
        \|r(y_k)\| \leq \|r(\hat y_k)\| = \|r(y(t_k)+O(t^{2k}))\| \leq O(\tau^{p+1}+\tau^{2k}). 
    \]
    Using mean-value theorem for vector-valued functions \cite{Ortega} it can be shown that for sufficiently smooth $f$ there exists $C<\infty$ such that
    \[
        \|y_k-y(t_k)\|\leq C \|r(y_k) - r(y(t_k))\|\leq O(\tau^{p+1}+\tau^{2k}).  
    \]
   \qed
 \end{proof}

\begin{example}\label{ex:order}
    To support this result we made a computational experiment on the linear model equation
     \begin{equation}\label{eq:linear-order-test}
         y_i'=\lambda_i y_i+1,\quad y_i(0)=1,
     \end{equation}
    where $\lambda_i$ are equally spaced on $[-\lambda_{\max},0]$, $n=100$, $t\in[0,1]$. We take $\lambda_{\max}=100$ and considered MRMS($k$, $p$) methods for $p=1,\ldots 7$ and the implicit Euler method. The standard convergence diagrams are shown on Figure \ref{fig:order-1}: the absolute error in endpoint $err$ is measured in the maximum-norm, the number of equal steps $N_{steps}$ is changing from $2^4$ to $2^{13}$ . In order to show the importance of parameter $k$ two diagrams are generated, for $k=p$ and $k=p+1$. The starting values for $k>1$ are taken from the exact solution.
    \begin{figure}
        \includegraphics[width=13cm]{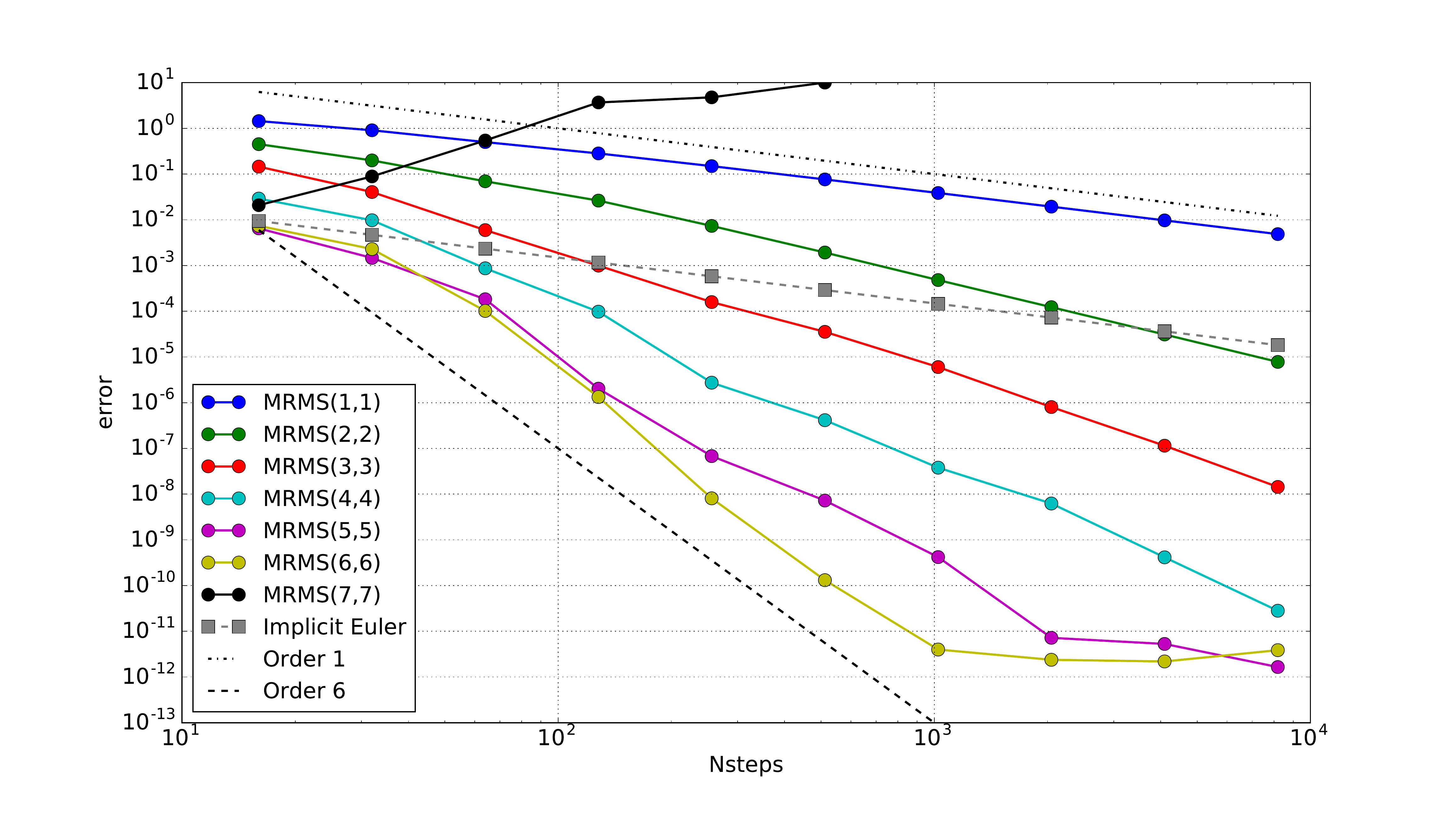}\\
        \includegraphics[width=13cm]{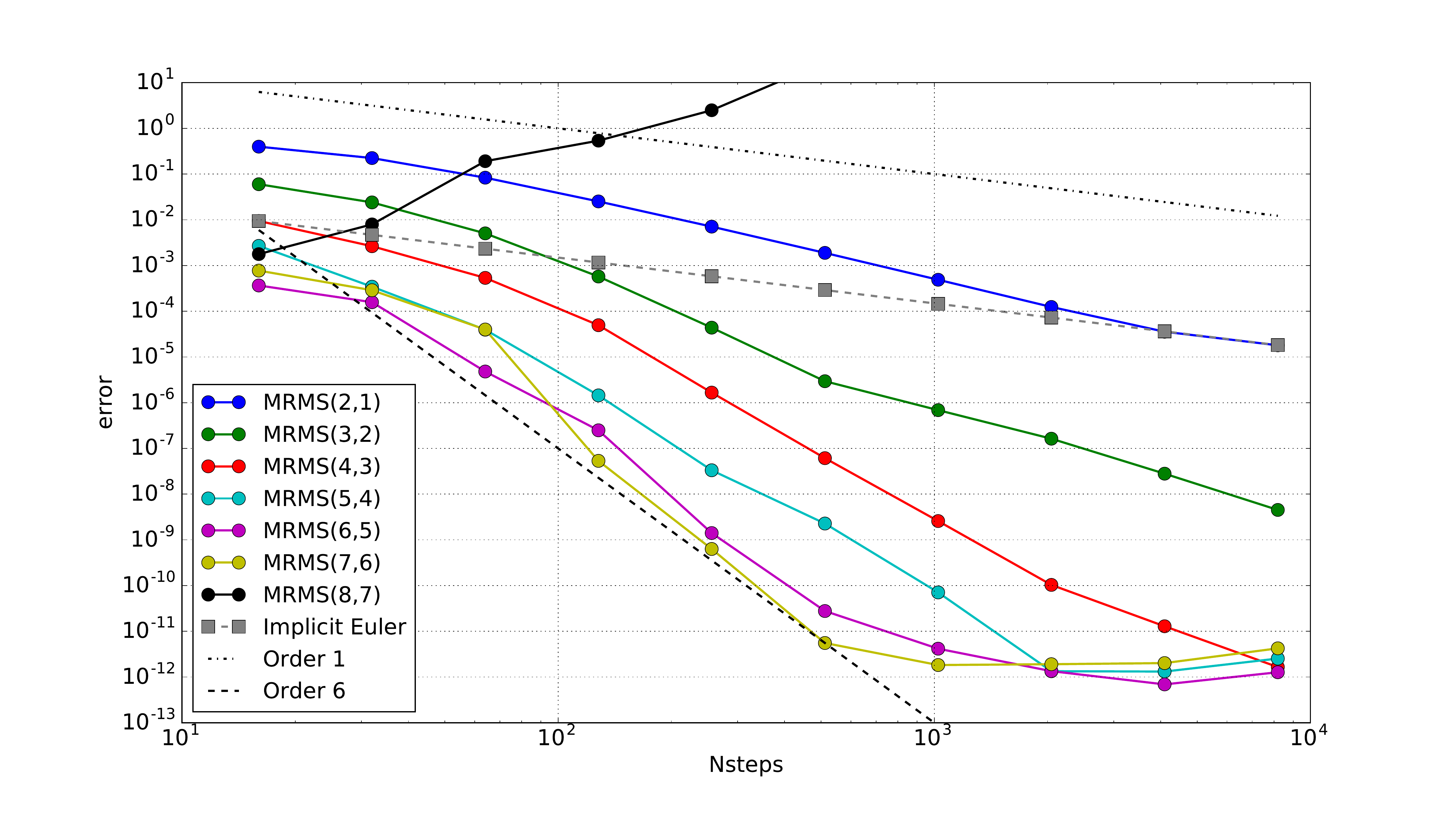}\\
        \caption{Convergence diagram from example \ref{ex:order}.}
    \end{figure}
    \label{fig:order-1}
\end{example}

\section{Zero-stability}
\label{sec:zero}

\begin{lemma}
    Let $t_j=j\tau$. If $p\leq k$ then a MRMS($k$, $p$) method inherits zero-stability from the underlying BDF method. 
\end{lemma}
\begin{proof}
    Since we are considering $f(t,y)=0$, the MRMS solution is 
    \[
    y_k=-(\alpha_0 y_0+\ldots+\alpha_{k-1} y_{k-1}),
    \]
    where $\alpha_0, \ldots, \alpha_{k-1}$ minimize the norm of
    \[
        -r(y_k)=c_k \sum_{j=0}^{k-1}(-\alpha_j y_j) + c_{k-1}y_{k-1}+ \ldots + c_{k-p}y_{k-p} = \sum_{j=0}^{k-1}(c_j-c_k \alpha_j) y_j,
    \]
    where we put $c_0=\ldots=c_{k-p-1}=0$. It is clear that optimal $r(y_k)$ is zero and there can be infinitely many vectors $\alpha$ for which this minimal value is reached. But for all these $\alpha$ we have
    \[
        y_k=-\sum_{j=0}^{k-1}\alpha_j y_j = -c_k^{-1}\sum_{j=0}^{k-1}c_j y_j,
    \]
    so the corresponding generating polynomial $\rho$ is 
    \[
        \rho(z)=\sum_{j=0}^{k-1}\alpha_j z^j + z^k = c_k^{-1}\sum_{j=0}^{k-1}c_j z^j + z^k = c_k^{-1} z^{k-p}\rho^{\mathrm{BDF}}(z),
    \]
    where $\rho^{\mathrm{BDF}}$ is the generating polynomial of degree $p$ of the underlying BDF method \eqref{eq:bdf}.\qed
\end{proof}


\section{Linear stability}
\label{sec:linear-stab}

The analysis of linear stability we are going to perform on the simplest case of $k=1$, $p=1$.

\subsection{Minimal residual Euler method}

The simplest method of MRMS family is MRMS(1,1) --- the analog of the Euler method:
\begin{equation}\label{eq:mrEuler}
    y_1=\alpha y_0 + \tau \beta f_0. 
\end{equation}
Note that here for clarity we put $\alpha=-\alpha_0$, $\beta=\beta_0$, which is not perfectly consistent with the previously introduced general notation. On each step these two coefficients are forced to minimize the norm of implicit Euler's residual:
\begin{equation}\label{eq:mre-optimize}
  (\alpha,\beta)=\argmin_{\alpha',\beta'}\|r_1(\alpha' y_0+\tau \beta' f_0)\|,  
\end{equation}
where
\begin{equation}
    r_1(x)=y_0+\tau f(t_0+\tau,x) - x.
\end{equation} Hereinafter we will call this method the minimal residual Euler (MRE) method. By theorem \ref{thm:mrms-order} the order of the method is equal to one. 

 

It is straightforward to check that for the standard linear test equation ${y'=\lambda y}$ the MRE method is equivalent to the implicit Euler method\footnote{Since we have two free parameters, in this case optimization problem \eqref{eq:mre-optimize} generally has infinitely many solutions. But all of them yield a solution $y_1$ which coincides with that of implicit Euler's method.}. Therefore to analyze linear stability of the method we must consider a multidimensional analogue:
\begin{equation}\label{eq:model-lin-sys}
\begin{array}{cc}
    y'=\Lambda y,\quad y(0)=y_0,\\
    \Lambda=\diag\{\lambda_1,\ldots,\lambda_n\},\quad
    y_0=(\eta_1,\ldots,\eta_n)^T.
\end{array}
\end{equation}

By definition of the MRE method in this case we have
\begin{equation}\label{eq:mre-Rz}
    y_1=R(\tau \Lambda) y_0,\quad R(z)=\alpha+\beta z,
\end{equation}
and
\begin{equation}
    r_1(y_1) = y_0+\tau \Lambda y_1 - y_1 = P(\tau \Lambda) y_0,
\end{equation}
\begin{equation}\label{eq:mre-p}
    P(z) = 1+(z-1)R(z).
\end{equation}
Optimization problem \eqref{eq:mre-optimize} now takes the form
\begin{equation}\label{eq:mre-ls}
    (\alpha,\beta)^T = \argmin_{\gamma\in\RR^2}\Bigl\|
    W \gamma + y_0
    \Bigr\|, 
\end{equation}
where
\begin{equation}
    W = \begin{bmatrix}
    \eta_1(z_1-1) & \eta_2(z_2-1) & \ldots & \eta_n(z_n-1)\\
    \eta_1 z_1(z_1-1) & \eta_2 z_2(z_2-1) & \ldots & \eta_n z_n(z_n-1)\\
    \end{bmatrix}^T,  
\end{equation}
$z_i=\tau \lambda_i$ and by $\eta_i$ we denote the components of $y_0$. In terms of polynomial \eqref{eq:mre-p} problem \eqref{eq:mre-ls} can be reformulated in the following way: find polynomial $P\in\PP^1$ such that
\begin{equation}\label{eq:mre-P-opt}
    \|P(\tau \Lambda)y_0\|=\min_{p\in\PP^1}\|p(\tau \Lambda)y_0\|.
\end{equation}
Here $\PP^1$ is the set of all polynomials $p$ of degree 1 such that $p(1)=1$. 

It should be noted that the MRE method is nonlinear: even on linear system the evolution operator $R(\tau A)$ depends on the initial condition $y_0$. Hence it is difficult to directly generalize the classical notion of stability region to the MRE case. The same is true for all MRMS methods in general.

Regarding absolute stability we are interested in considering $z_i\leq 0$. For definiteness let's assume that
\[0 \geq z_1 \geq z_2 \geq \ldots \geq z_n. \]
We need to check the conditions
\begin{equation}
    |R(z_i)|\leq 1.
\end{equation}
To analyze the simplest case when all $|\eta_i|$ are equal it is reasonable to correlate polynomial $P$ from \eqref{eq:mre-P-opt} with the minimax polynomial $\tilde P\in \PP^1$,
\begin{equation}
    \tilde P=\argmin_{p\in\PP^1}\max_{z\in [z_n,0]}|p(z)|.
\end{equation}
As is known (see, e.g. \cite{Saad}) $\tilde P$ can be expressed using Chebyshev polynomial $T_2$ as
\begin{equation}\label{eq:tilde-P}
    \tilde P(z)=\frac{T_2\left(1-\dfrac{2z}{z_n}\right)}{T_2\left(1-\dfrac{2}{z_n}\right)}=\frac{8z^2-8z_n z+z_n^2}{z_n^2-8z_n+8}.
\end{equation}
The maximum deviation of $\tilde P(z)$ from zero is equal to 
\begin{equation}\label{eq:mre-tilde-eps}
    \tilde \epsilon =  \frac{z_n^2}{z_n^2-8z_n+8}.
\end{equation}


\subsubsection{Maximum-norm minimization}

In the case $\|\cdot\|=\|\cdot \|_\infty$ the use of \eqref{eq:tilde-P} gives us the following result.
\begin{proposition}\label{prop:mre-maxnorm}
 Consider minimization problem \eqref{eq:mre-P-opt} in maximum-norm. If all $|\eta_i|$ are equal, then 
\begin{equation}
    0 < R(z_i)<\frac{2}{1-z_i}\quad \forall\: i=1,\ldots,n.
\end{equation}
\end{proposition}
\begin{proof}
    By definition
    \begin{equation*}
        \|P(\tau\Lambda)y_0\|_\infty = \max_i|P(z_i)\eta_i| \leq \|\tilde P(\tau\Lambda)y_0\|_\infty = \max_i|\tilde P(z_i)\eta_i|\leq \tilde\epsilon |\eta_1|,
    \end{equation*}
    where $\tilde\epsilon$ is defined in \eqref{eq:mre-tilde-eps}. From this it follows that for any $i$
    \begin{equation*}
        |P(z_i)|=|1+(z_j-1)R(z_i)|\leq \tilde\epsilon<1.
    \end{equation*}
    \qed
\end{proof}
\begin{corollary}
    If conditions of proposition \ref{prop:mre-maxnorm} hold then $R(z_i)\in(0,1)$ for all $z_i\leq -1$.
\end{corollary}

While in practice we use the 2-norm minimization, this simple result already shows that potentially the MRE method can be stable on stiff problems. 

\subsubsection{Minimization in 2-norm}

Firstly let's ensure that the MRE method is well-defined when minimization \eqref{eq:mre-ls} is held in $\|\cdot\|_2$. 
Problem \eqref{eq:mre-ls} is well-posed when columns of $W$ are linearly independent. By considering all minors of order 2 we get the following necessary and sufficient condition.
\begin{proposition}\label{prop:mre-stab-well-posed}
    Linear least squares problem \eqref{eq:mre-ls} is well-posed iff there exist $i\ne j$ such that $z_i\ne z_j\ne 1$ and $\eta_i \eta_j\ne 0$.
\end{proposition} 
It is important to realize that the solution of \eqref{eq:mre-ls} always exists, while well-posedness guarantees its uniqueness.

\begin{proposition}\label{prop:mre-ie-equi}
    Consider model linear system \eqref{eq:model-lin-sys} with time step $\tau$ and let the implicit Euler approximate solution $y_1^{IE}$ be well-defined, i.~e. $z_i=\tau\lambda_i\ne 1$ for all $i$. In the following cases the MRE solution $y_1^{MRE}$ coincides with $y_1^{IE}$.

    (i) Linear least squares problem \eqref{eq:mre-ls} 
    is well-posed, but the set $\{z_i\}$ contains only two distinct elements.
     
    (ii) Problem \eqref{eq:mre-ls} is not well-posed (the conditions of proposition \ref{prop:mre-stab-well-posed} are not fulfilled).  
\end{proposition}
\begin{proof}
    In case (i) let $z_1$ and $z_2$ be the two possible distinct values of $z_i$. We can always choose the polynomial $R$ from \eqref{eq:mre-Rz} such that $R(z_1)=(1-z_1)^{-1}$ and $R(z_2)=(1-z_2)^{-1}$, so $r_1(y_1^{MRE})=0$.
    Case (ii) can be easily proved in the similar manner.\qed
\end{proof}

It is interesting that the MRE optimization problem can be well-defined even if the implicit Euler method is not.
\begin{example}\label{ex:ie-fail}
    Consider $n=3$, $y_0=(1,1,1)^T$, $\lambda=\{-1,0,1\}$ and $\tau=1$ the MRE stability polynomial from \eqref{eq:mre-Rz} is $R(z)=1+\frac 12 z$. 
\end{example}

From proposition \ref{prop:mre-ie-equi} it follows that in case when the optimization problem \eqref{eq:mre-ls} is not well-posed, the MRE method appears to be absoltely stable. Hereinafter we are going to consider only well-posed least-squares problems.

Denoting \[s_i=1-z_i\] the normal equations for \eqref{eq:mre-ls} take the form
\begin{equation}\label{eq:mre-normal-eqs}
\renewcommand{\arraystretch}{1.5}
    \begin{bmatrix}
        \sum \eta_i^2 s_i^2 & \sum \eta_i^2  s_i^2 z_i\\
        \sum \eta_i^2 s_i^2 z_i & \sum \eta_i^2  s_i^2 z_i^2
    \end{bmatrix}
    \begin{bmatrix}
    \alpha\\ \beta
    \end{bmatrix}=
    \begin{bmatrix}
    \sum \eta_i^2  s_i\\
    \sum \eta_i^2  s_i z_i
    \end{bmatrix}
\end{equation}
and the solution is
\begin{equation}
    \alpha=\frac{\Delta_1}{\Delta},\quad \beta=\frac{\Delta_2}{\Delta},
\end{equation}
where $\Delta$ and $\Delta_i$ are usual Kramer determinants of \eqref{eq:mre-normal-eqs} in which the summation is held over all $i$ from $1$ to $n$.

\begin{lemma}\label{lem:mre-stab-positive}
    Consider the MRE method  \eqref{eq:mrEuler} applied to linear system \eqref{eq:model-lin-sys}. Suppose that the conditions of proposition \ref{prop:mre-stab-well-posed} are fulfilled and all $z_i\leq 0$.
    Then the coefficients $\alpha$ and $\beta$  
    satisfy $0 \leq\beta\leq\alpha\leq 1$.
\end{lemma}
\begin{proof}
    The proof is by induction on $n$. Let $\Delta^n$, $\Delta^n_1$ and $\Delta^n_2$ be the corresponding Kramer determinants of \eqref{eq:mre-normal-eqs}.
    After collecting similar terms these determinants can be written as
 \begin{subequations}\label{eq:mre-deltas}
       \begin{align}
            &\Delta^n=\displaystyle\sum_{i=1}^n\eta_i^2 s_i^2 z_i
                \displaystyle\sum_{j=1}^n\eta_j^2 s_j^2(z_i-z_j),\\
            &\Delta^n_1=\displaystyle\sum_{i=1}^n\eta_i^2 s_i^2 z_i
                \displaystyle\sum_{j=1}^n\eta_j^2 s_j(z_i-z_j),\\
            &\Delta^n_2=-\displaystyle\sum_{i=1}^n\eta_i^2 s_i^2
                \displaystyle\sum_{j=1}^n\eta_j^2 s_j(z_i-z_j).
        \end{align}
     \end{subequations}
     For $n=2$ we have
    \[
        0\leq \Delta^2_2\leq \Delta^2_1 \leq \Delta^2.
    \] 
    In general case by direct computation it can be shown that
    \[
        \Delta^n=\Delta^{n-1}+\delta^{n-1},\quad
        \Delta^n_1=\Delta^{n-1}_1+\delta^{n-1}_1,\quad
        \Delta^n_2=\Delta^{n-1}_2+\delta^{n-1}_2,
    \] 
    where
     \begin{align*}
        &\delta^{n-1}=\eta_n^2 s_n^2
            \displaystyle\sum_{i=1}^{n-1}\eta_i^2 s_i^2(z_i-z_n)^2,\\
        &\delta^{n-1}_1=\eta_n^2 s_n
            \displaystyle\sum_{i=1}^{n-1}\eta_i^2 (z_i-z_n)^2(1-z_i-z_n),\\
        &\delta^{n-1}_2=\eta_n^2 s_n
            \displaystyle\sum_{i=1}^{n-1}\eta_i^2 s_i (z_i-z_n)^2.
    \end{align*}   
Since 
\[
    0\leq \delta^{n-1}_2\leq \delta^{n-1}_1\leq \delta^{n-1}
\] 
for all $n\geq 2$, in general case we have  
    \[
        0\leq \Delta^n_2\leq \Delta^n_1 \leq \Delta^n 
    \] 
and $0 \leq\beta=\Delta^n_2/\Delta^n \leq \alpha = \Delta^n_1/\Delta^n\leq 1$.
\qed
\end{proof}

A consequense of this lemma is that $R(z_i)\leq 1$ for $z_i\leq 0$, and the stability condition now reduces to $R(z_i)\geq -1$. 
Since $R$ is a first-degree polynomial, if $R(z_n)\geq -1$ then we will have $R(z_i)\geq -1$ for all the rest $z_i$.

\begin{lemma}
Suppose that the conditions of proposition \ref{prop:mre-stab-well-posed} are fulfilled, all $z_i\leq 0$ and all components of $y_0$ have equal magnitudes: $|\eta_i|=|\eta_1|$. If 
\[R(z_n)=\alpha+\beta z_n < -1\]
then 
\begin{equation}\label{eq:Bzn}
    B_n(z_n)=(2-z_n)(z_n^2-8z_n+8)-\sqrt n z_n^2 < 0.
\end{equation}
\end{lemma}
\begin{proof}
    The proofs is similar to that of proposition \ref{prop:mre-maxnorm}. First of all by definition of $P$ we have 
    \begin{equation*}
        P(z_n) = 1+(z_n-1)R(z_n) > 2-z_n    
    \end{equation*}
    and 
    \begin{equation*}
        \|P(\tau\Lambda)y_0\|^2_2>\eta_1^2\left(
        (2-z_n)^2+\sum_{i-1}^{n-1}P(z_i)^2
        \right).
    \end{equation*}
    On the other hand,
    \begin{equation*}
        \|P(\tau\Lambda)y_0\|_2^2\leq \|\tilde P(\tau\Lambda)y_0\|_2^2\leq n \eta_1^2 \tilde\epsilon^2.
    \end{equation*}
    Combining these two inequalities and using \eqref{eq:mre-tilde-eps} we have
    \begin{equation*}
        (2-z_n)^2< n \left(\frac{z_n^2}{z_n^2-8z_n+8}\right)^2.
    \end{equation*}
    \qed
\end{proof}
A simple analysis of polynomial $B$ from \eqref{eq:Bzn} yields that
\begin{itemize}
    \item $B_n(z)$ is positive for all $z\leq 0$ if $\sqrt n < 2(5+3\sqrt 3)\approx 20.3923$;
    \item For $\sqrt n > 2(5+3\sqrt 3)$ $B_n$ has two real negative roots $a_n, b_n$ such that $B_n(z)<0$ for $z\in(a_n,b_n)$;
    \item The leftmost root $a_n$ lies between $10-\sqrt n$ and $11-\sqrt n$.
\end{itemize}
\begin{corollary} \label{cor:mre-stab}
    Let the conditions of proposition \ref{prop:mre-stab-well-posed} hold, all $z_i\leq 0$ and all $|\eta_i|$ are equal. Then $|R(z_i)|
    \leq 1$ for all $i=1,\ldots n$ in the following cases:

    (i) $n\leq 415$;

    (ii) $n>415$ and $z_n \leq 10-\sqrt n$.        
\end{corollary}
Note that the requirements of corollary \ref{cor:mre-stab} are not necessary, but sufficient for stability. 

\subsection{Instability cases}\label{ss:instab}

Though the result of corollary \ref{cor:mre-stab} is quite optimistic we should always keep in mind the dependence of the stability polynomial $R$ on the values of $\eta_i$. In fact it is quite easy to construct a case when $|R(z_n)|>1$. It suffices to consider $\eta_n=0$: now the $n$-th component of $P(\tau \Lambda)y_0$ is always zero and $P(z_n)$ does not influence the functional that is being minimized. Taking into account lemma \ref{lem:mre-stab-positive} this means that when $\eta_n=0$ and $z_n \ll z_{n-1}$ then $R(z_n)\ll -1$.
On the other hand this is not a big problem, since this huge number will be multiplied by zero. But by continuity argument in such a case we will get $R(z_n)<-1$ also for all $\eta_n$ in a neighborhood of zero. Anyway the instability in this case does not occur for big values of $\eta_i$ and thus can not lead to a blow-up. The following example illustrates the described situation.
\begin{example}\label{ex:instab}
    Let $n=3$, $y_0=(1,1,\eta)^T$, $z_1=0$, $z_2=-1$, and $z_3$ is a parameter. Then
    $$R(z_3)=
    \dfrac{2 (z_3+2)-\eta ^2 (z_3-1) (5 z^2_3+8z_3+4)}
    {\eta ^2 (z_3-1)^2 (5 z_3^2+8z_3 +4)+4},
    $$
     so at $\eta=0$ we have $R(z_3)=1+z_3/2$. Graphs of $-R(z_3)$ are shown at the first row of Figure \ref{fig:example2}. They display that $R(z_3)$ tends to 0 very fast as $\eta$ grows, but remains negative in a neighborhood of 0. This will result in a fading wiggling around equilibrium in an approximate solution as $t\to\infty$. 

     Another important observation is the behavior of $R(0)$ which is equal to
     \begin{equation}
        R(0)=\frac{\eta ^2 ((z_3 -2) z_3  (z_3 +1) (3 z_3 -1)+4)+4}{\eta ^2 (z_3 -1)^2(5 z_3^2+8z_3 +4) +4}.
    \end{equation}
    The second row of Figure \ref{fig:example2} shows how $R(0)$ depends on $\eta$ for $z_3=-10$ and $z_3=-10^6$. We see that $R(0)$ can be much less than 1. This means that generally very slow eigenmodes of the solution are heavily damped by the MRE (lemma \ref{lem:mre-stab-positive} confirms this statement). This is the main reason of poor accuracy of the MRE method in our first experiment from Figure \ref{fig:order-1}. 

     To sum up these observations consider $\eta=1$, $\lambda_3=-10^5$, $\tau=0.05$ and compare the MRE method and the implicit Euler method on $[0,1]$. The results are shown on Figure \ref{fig:example2-2}. 
    \begin{figure}\label{fig:example2}
    \includegraphics[width=14cm]{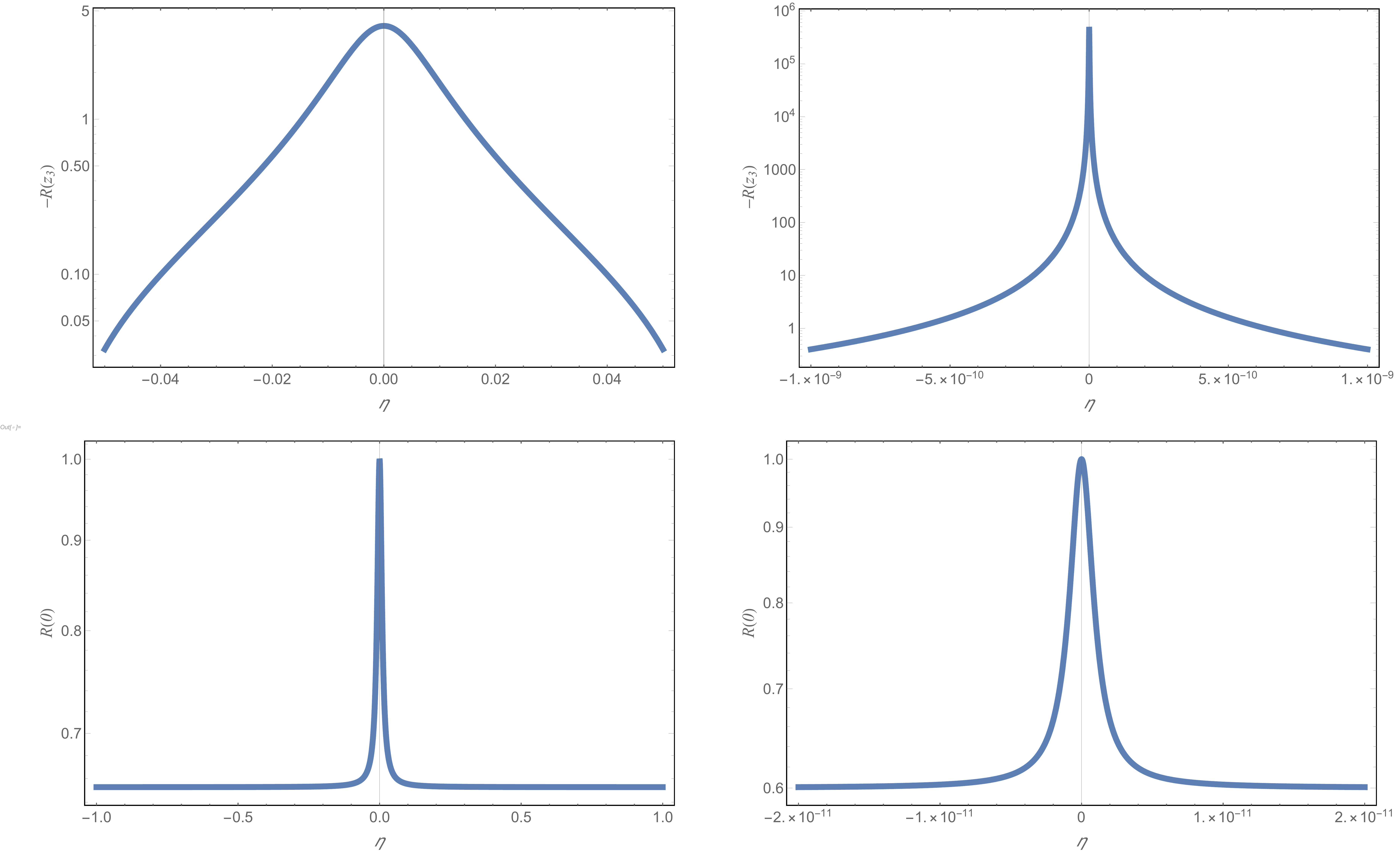} 
    \caption{Illustration of example \ref{ex:instab}. Graphs of dependence of $-R(z_3)$ (top row) and $R(0)$ (bottom row) on $\eta$ for $z_3=-10$ (left column) and $z_3=-10^6$ (right column).}
    \end{figure}

    \begin{figure}\label{fig:example2-2}
    \begin{tabular}{ccc}
        \includegraphics[width=6cm]{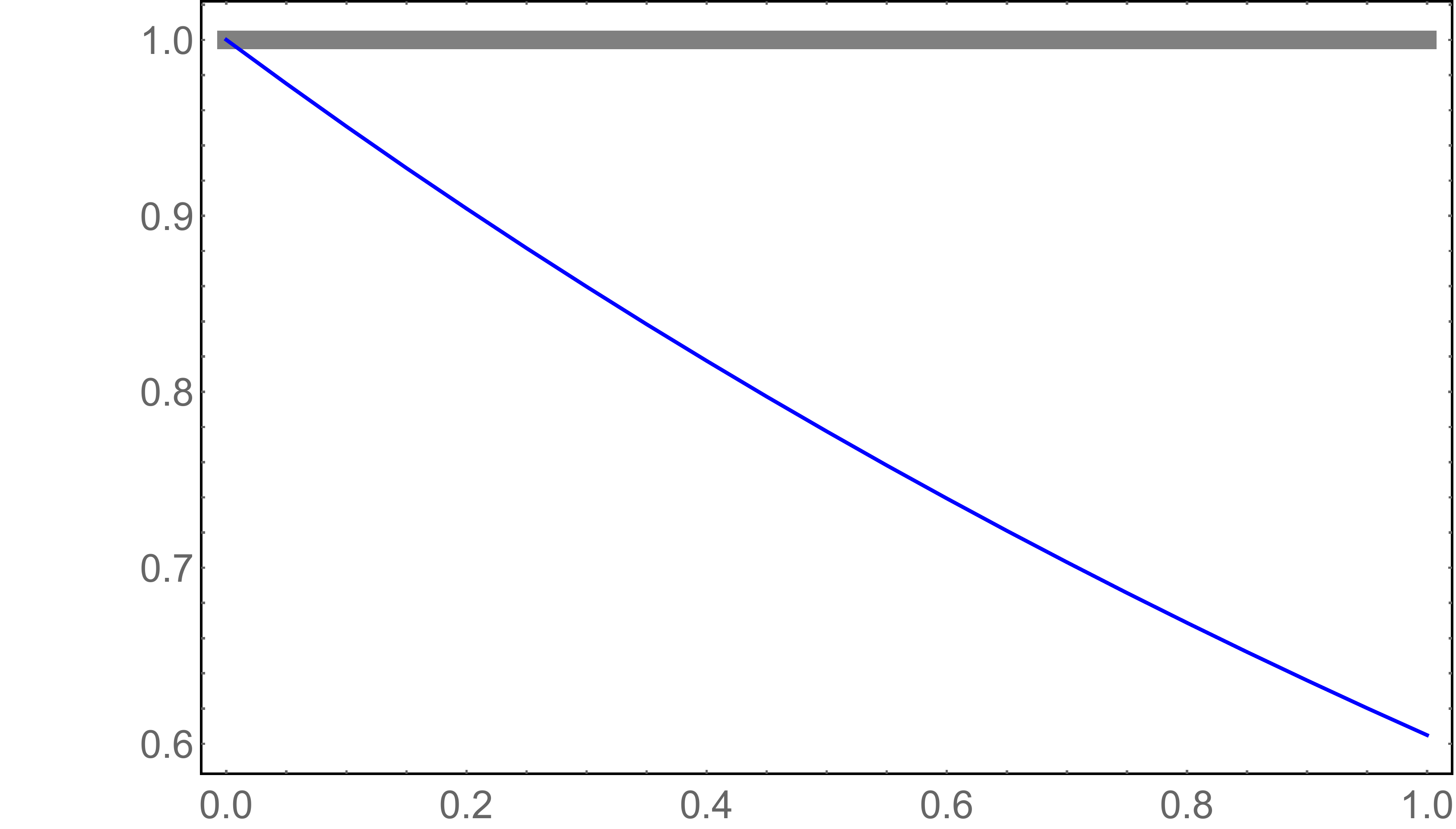} 
        \includegraphics[width=6cm]{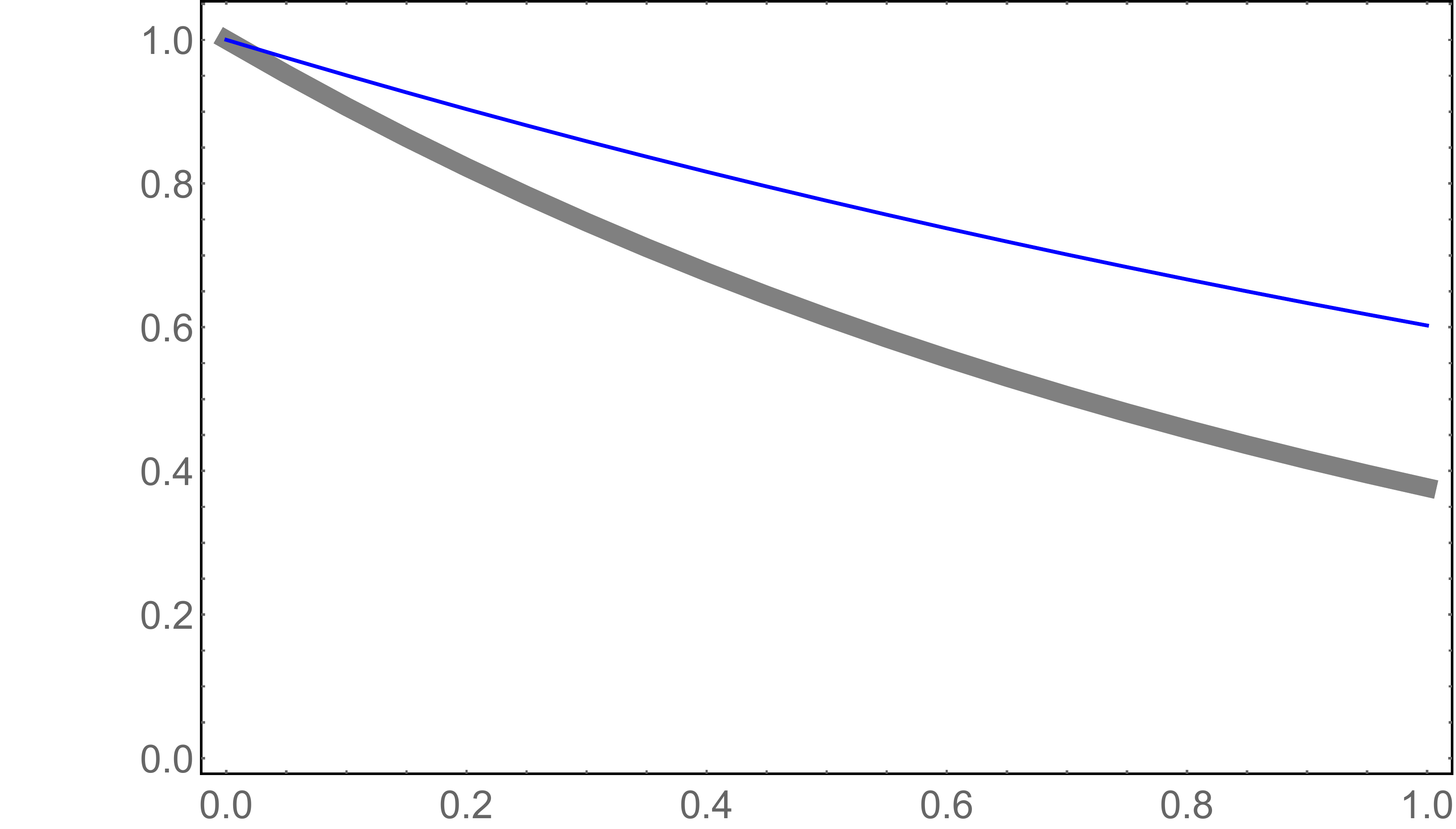} \\
        \includegraphics[width=6cm]{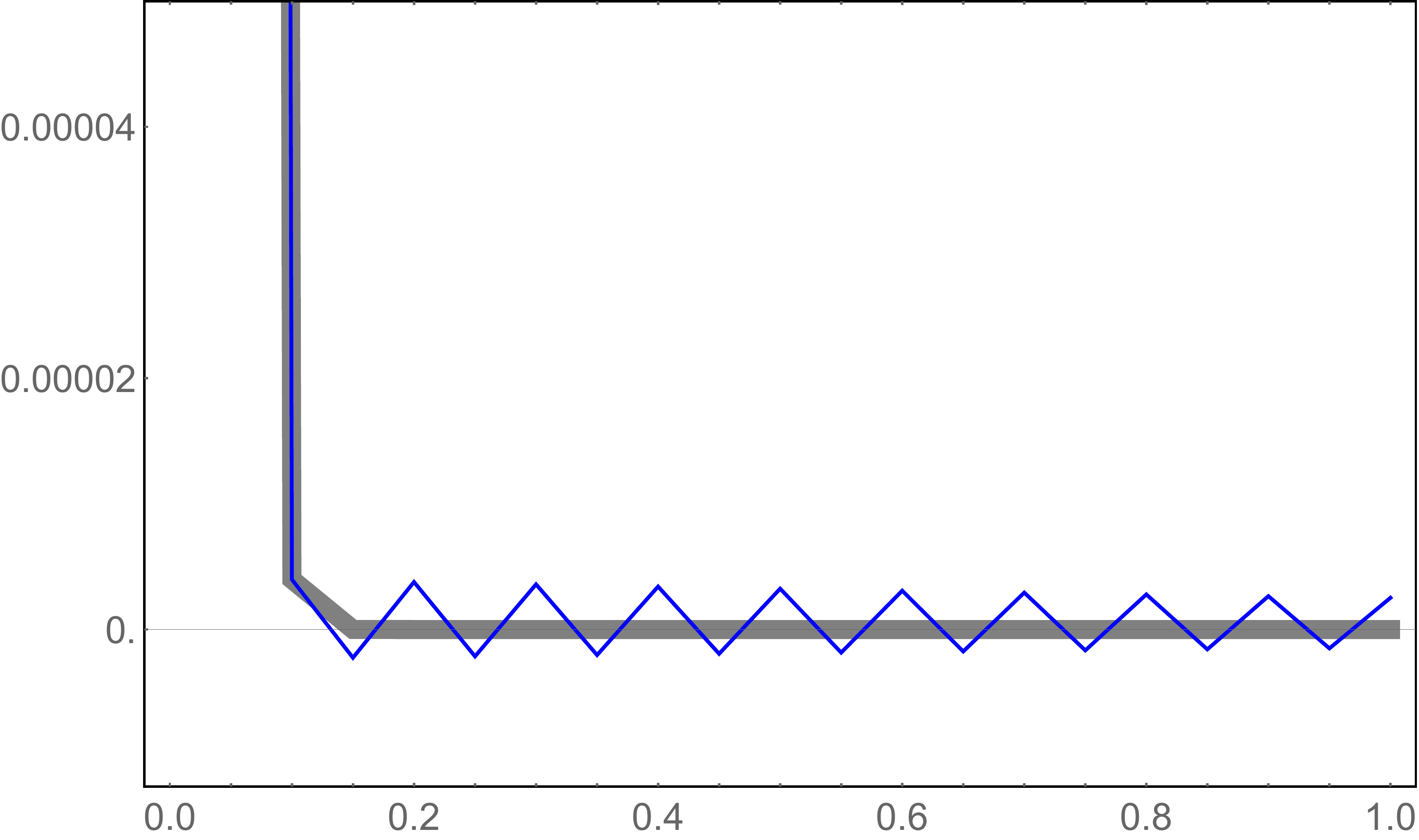}     
    \end{tabular}
    \caption{Numerical results for example \ref{ex:instab}. Graphs of approximations to solution of the system $y_1'=0, y_2'=-y_2, y_3'=-10^5, y_i(0)=1$, by the MRE method (thin blue line) and the implicit Euler method (thick gray line) are shown. Constant time step $\tau=0.05$ is used.}
    \end{figure}
\end{example}

\subsection{Linear stability in general case}
\label{ss:gen-lin-stab}

The dependence of the MRMS coefficients on the values of the solution makes the general analysis of linear stability a difficult problem. If we try to follow the common scheme of absolute stability analysis for linear multistep methods \cite[V.1]{hw-vol2}, then we need to analyze the location of roots of equation
\begin{equation}
     \rho(\zeta) - \mu\sigma(\zeta) = 0,
 \end{equation} 
 where $\rho(\zeta)=\zeta^k+\displaystyle \sum_{j=0}^{k-1}\alpha_j \zeta^j$, $\sigma(\zeta)=\displaystyle \sum_{j=0}^{k-1}\beta_j \zeta^j$. Just as with the MRE method two issues complicate things: (i) the coefficients $\{\alpha_j\}$ and $\{\beta_j\}$ are nonlinear functions of $\{y_j\}$; (ii) a set of values for parameter $\mu$ must be considered at once, and the distribution of these values must be taken into account. It is not very clear how to deal with this problem, so we leave it for future research.

 To see how much can we ever expect let's repeat the experiment from example \ref{ex:order} and Figure \ref{fig:order-1} with $\lambda_{\max}=10^7$. Recall that the system $y'_i=\lambda_i y + 1$, $i=1, 2, \ldots, 100$ is solved, where $\lambda_i$ are evenly distributed on $[-\lambda_{\max}, 0]$ (including the endpoints). For such a problem implicit Euler method readily jumps to the equilibrium. We consider this method, the MRMS($p$,$p$) and the MRMS($p+1$,$p$) methods on Figure \ref{fig:diag-2}. 

Now let's make an important modification of the problem: instead of uniform distribution let $\lambda_i = -(10^{m_i})$, where $m_i$ are equispaced on $[-7, 7]$. Thereby now we have more eigenvalues near zero, while the maximum negative eigenvalue is the same as before. The corresponding diagrams are shown at Figure \ref{fig:diag-3}. Note that there we displayed the results of MRMS($p$,$p$) and  MRMS($p+4$,$p$). 

\begin{figure}
\begin{tabular}{ccc}
    \includegraphics[width=13cm]{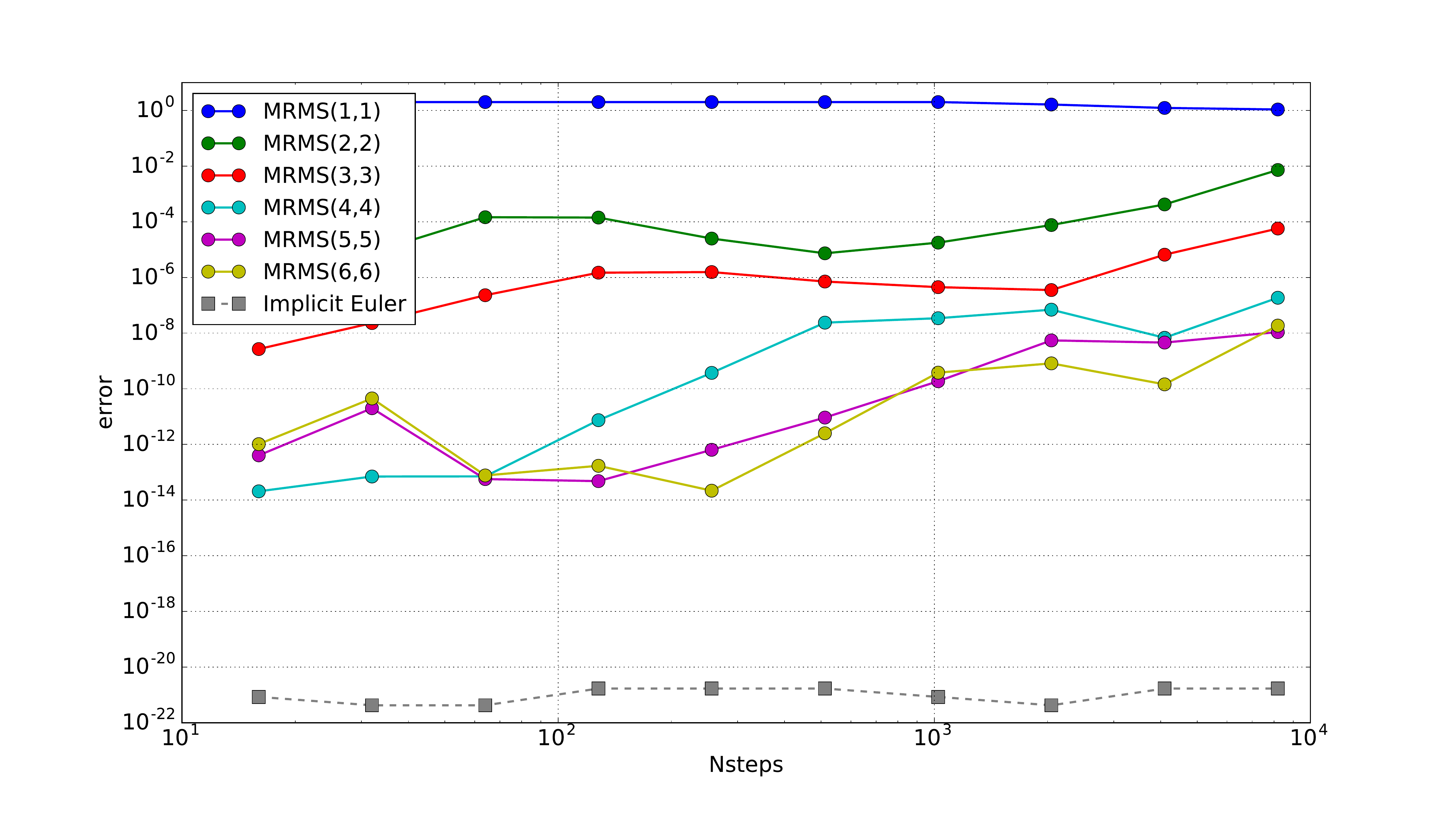} \\
    \includegraphics[width=13cm]{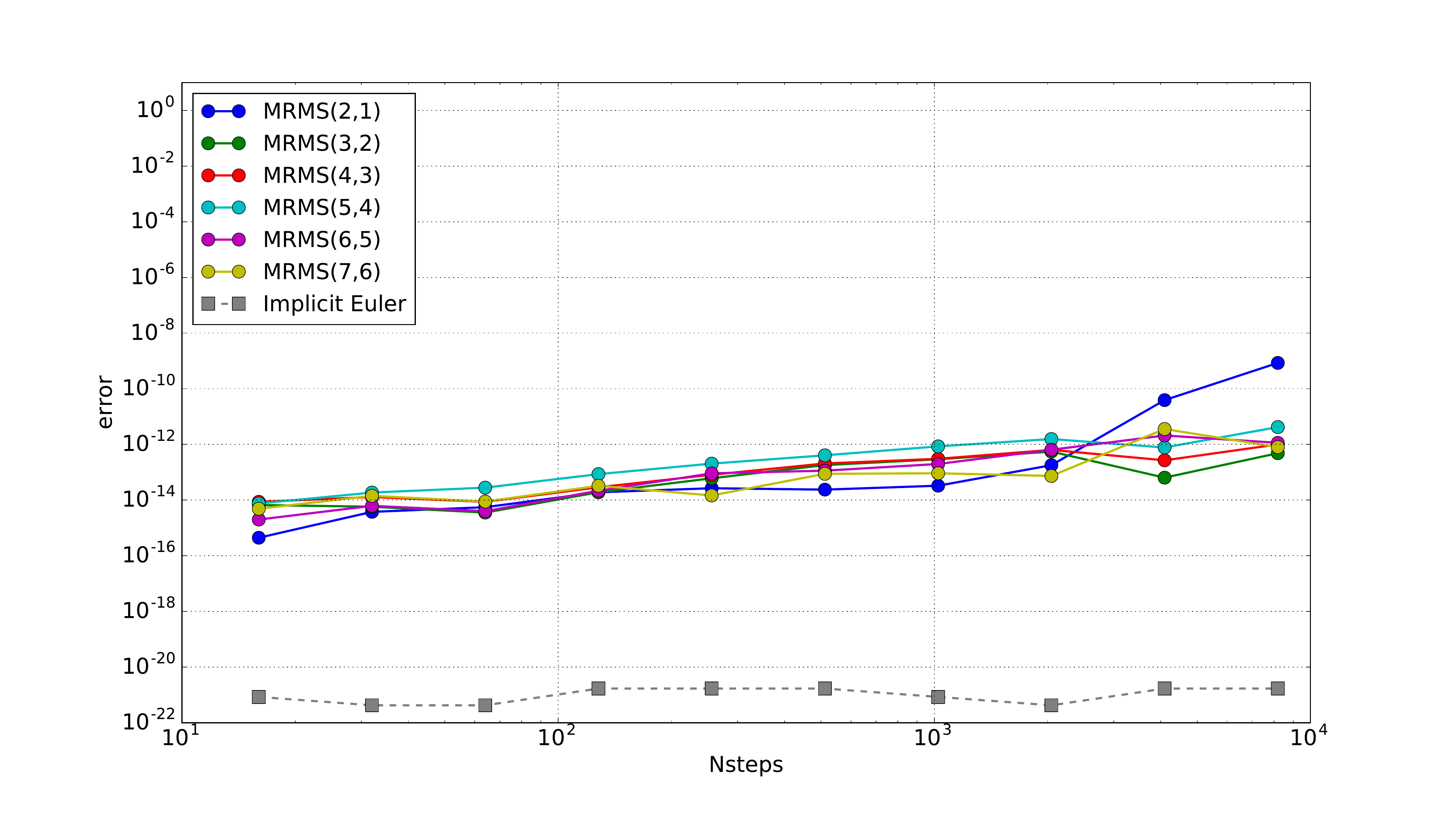} \\
\end{tabular}
\caption{The results of numerical experiment with stiff linear model problem \eqref{eq:model-lin-sys} from subsection \ref{ss:gen-lin-stab}. The eigenvalues of the diagonal matrix of ODE system \eqref{eq:linear-order-test} are uniformly distributed on $[-10^6, 0]$.}
\label{fig:diag-2}
\end{figure}

\begin{figure}
\begin{tabular}{ccc}
    \includegraphics[width=13cm]{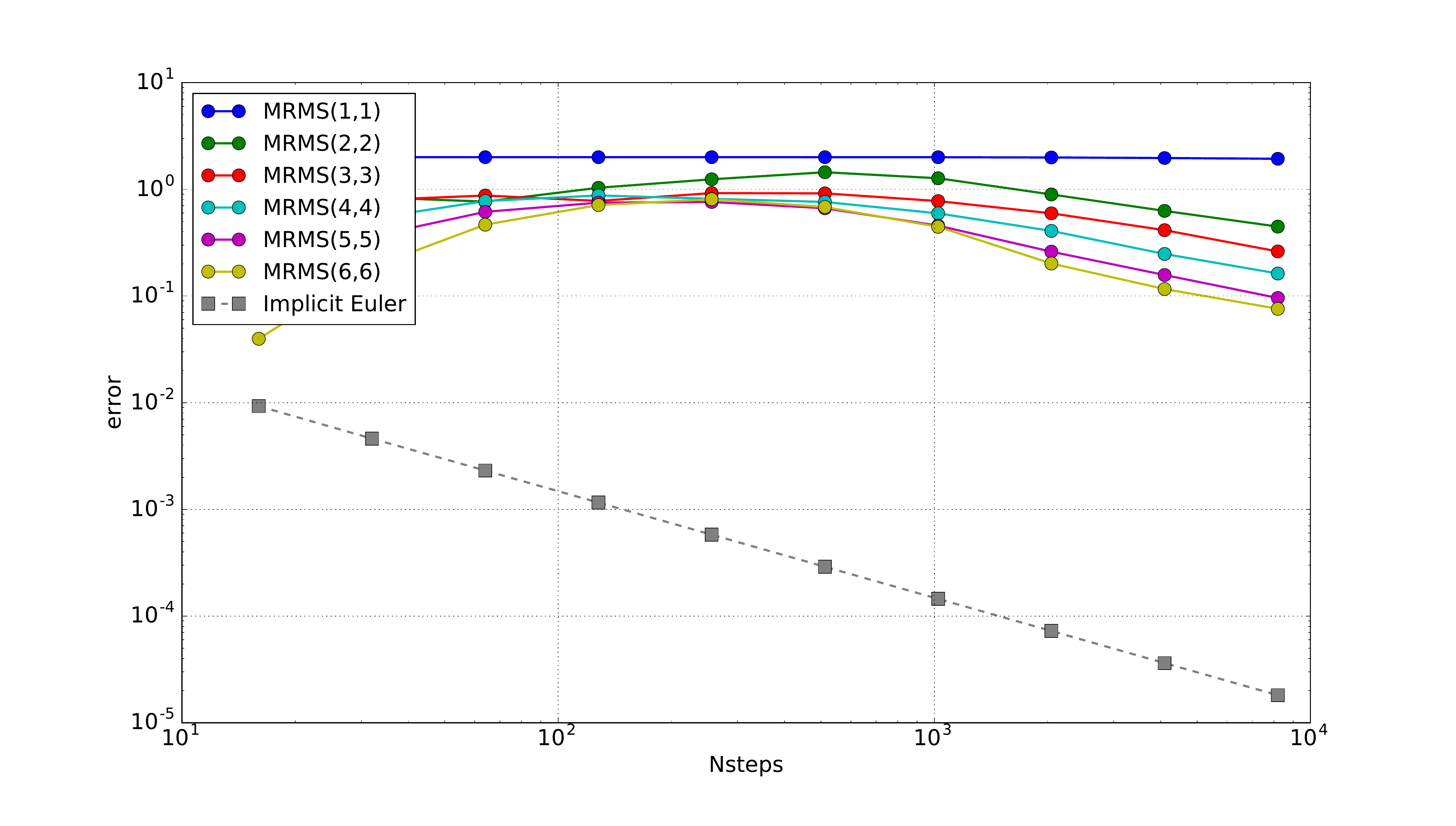} \\
    \includegraphics[width=13cm]{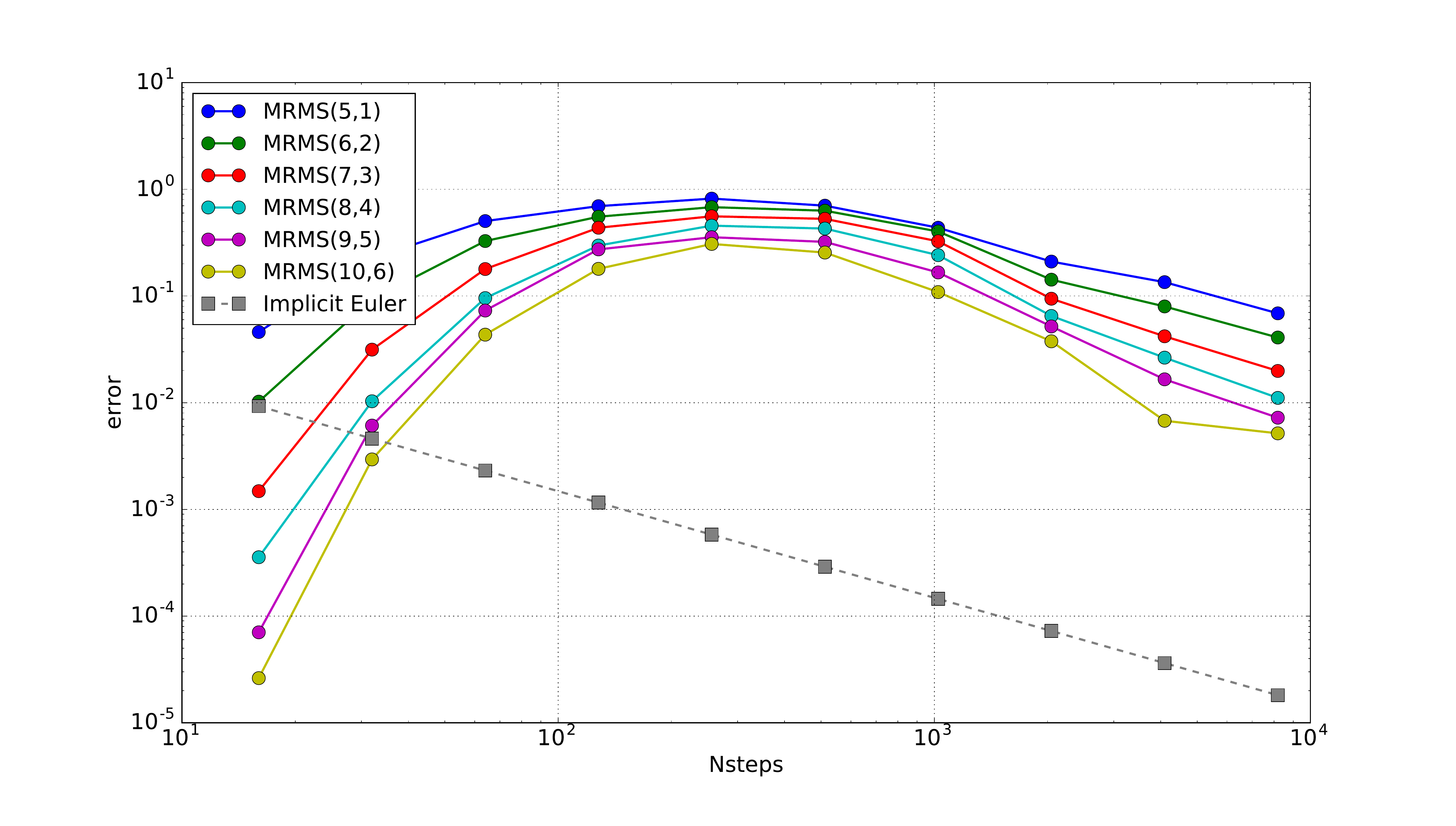} \\
\end{tabular}
\caption{The results of numerical experiment with stiff linear model problem \eqref{eq:model-lin-sys} from subsection \ref{ss:gen-lin-stab}. The eigenvalues of the diagonal matrix of ODE system \eqref{eq:linear-order-test} are $-(10^{m_i})$, where $m_i$ are uniformly distributed on $[-7, 7]$.}
\label{fig:diag-3}
\end{figure}

 Now we are ready to comment on the results of the experiment:
\begin{enumerate}
    \item No instability is observed for $p<7$ as it would happen in case of explicit methods.
    \item The accuracy gets significantly improved as additional vectors are added to the subspace $\mcV$.
    \item The MRMS methods exhibit the non-monotonic error decay as the step size decreases. Taking into account the stability analysis of the MRE method described above we can suppose that the reason is the sensitivity to the number of $z_i$ close to 0, rather than a numerical instability issue.
    \item The accuracy significantly depends on the distribution of the Jacobian matrix eigenvalues. More research is needed to understand this effect better.  
\end{enumerate}

\section{Implementation in the linear case}
\label{sec:implementation}

The most tractable class of problems for MRMS methods are non-autonomous linear systems of the form
\begin{equation}\label{eq:linear-system}
    y'(t)=f(t,y(t))=A(t)y(t)+b(t),
\end{equation}
where $A(t)$ is a $n\times n$ matrix, $b:\RR\to\RR^n$. Then BDF residual function \eqref{eq:r} is
\begin{equation}\label{eq:r-lin}
    r(x)=(\tau A(t_k)-c_k I)x + \tau b(t_k) - c_{k-1} y_{k-1} - \ldots - c_{k-p}y_{k-p},
\end{equation}
the approximate solution is 
\begin{equation*}
    y_k=\sum_{j=0}^{k-1}(\tau \beta_j f_j-\alpha_j y_j) 
\end{equation*}
and the optimization problem is to find
\begin{equation}\label{eq:ls-general}
    (\alpha,\beta)^T=\argmin_{\gamma\in\RR^{2k}}\|W\gamma-g\|,
\end{equation}
where
\begin{equation}\label{eq:ls-Wg}
    W=(\tau A(t_k)-c_k I)V,\quad g=\sum_{j=1}^p c_{k-j} y_{k-j}-\tau b(t_k).
\end{equation}
Here $V$ is a $n\times 2k$ matrix defined by \eqref{eq:Vk}. 

In our experiments we always solve problem \eqref{eq:ls-general} in the 2-norm. If we write the solution of this problem using the pseudo-inverse matrix $W^+=(W^T W)^{-1}W^T$ then the sought approximate solution $y_k$ can be also written as
\begin{equation}
    y_k=V W^+ g. 
\end{equation}

\subsection{Computational complexity}

The most computationally complex stages of one MRMS($k$,$p$) step are the following.
\begin{itemize}
     \item The computing of $A(t_k)$ and $b(t_k)$ --- at most $O(n^2)$ flops in the case of dense matrix.
     \item The computing of matrix $W = \tau A(t_k)V-c_k V$, which approximately amounts to $2k$ multiplications of matrix $A(t_k)$ by vectors $y_0,\ldots,y_{k-1}$, $f_0,\ldots,f_{k-1}$ which constitute $V$. However if $A$ and $\tau$ are constant then the most of columns in matrix $A V$ will be available from the previous steps and only two matrix-vector multiplications are actually required.  
     \item The solution of the linear least squares problem \eqref{eq:ls-general}. This problem can be solved in many ways, but the most widely used methods are the normal equations, the QR decomposition and the singular value decomposition. In the case $n \gg k$, which we are mostly concentrated on, the asymptotic complexity of these algorithms is $O(n k^2)$ flops \cite{demmel97}. Again, if $A$ and $\tau$ are constant, then matrices $W$ on two subsequent steps differ only in two columns. This allows reusing the results of QR decomposition and update it in $O(n k)$ flops. We do not use this opportunity in the forthcoming numerical experiment, though.  
 \end{itemize} 
Therefore, in the dense case the asymptotic complexity of each step is $O(n^2)$, which is to be compared with $O(n^3)$ required by matrix factorization within the standard implicit methods. But it should be realized that LU decompositions in standard implicit solvers are not computed on each step, and if it is already available, then the cost of one `implicit' step is $O(n^2)$ as well. 

\section{Numerical experiment}
\label{sec:experiment}

\subsection{The problem}

For the numerical experiment we consider the two-dimensional heat equation on the unit square:
\begin{gather}
	\frac{\partial u}{\partial t} = \frac{\partial^2 u}{\partial x^2} +  
	\frac{\partial^2 u}{\partial y^2} + f(x, y, t),\quad
	(x, y) \in S = (0,1)\times(0,1),\\
	u(x, y, t)\Bigr|_{(x,y)\in\partial S}=0,\quad u(x, y, 0) = u_0(x,y).
\end{gather}
We solve this problem using the standard method of lines using the five-point stencil for the discretization of Laplacian on the uniform mesh $(x_i, y_j) = (i h, j h)$, $h = 1/(N+1)$:
\begin{equation}\label{eq:heat-mol}
 	w_{ij}'(t)= \frac{1}{h^2}\Bigl(w_{i,j+1} + w_{i,j-1} + w_{i+1,j} + w_{i-1,j} - 4 w_{ij}\Bigr) + b_{ij}(t),
\end{equation}
$i,j = 1,2,\ldots N$, $w_{ij}(t) \approx u(x_i, y_j, t)$, $b_{ij}(t) = f(x_i, y_j, t)$.
In order to exactly compute the error of our methods  we construct a problem with preset solution of the form 
\begin{equation}\label{eq:exact_sol}
	w_{ij}(t)=p(t)q_{ij}
\end{equation}
by defining $b(t)$ as 
\begin{equation*}
	b_{ij}(t)=p'(t)q_{ij} - \frac{p(t)}{h^2}\Bigl(q_{i,j+1} + q_{i,j-1} + q_{i+1,j} + q_{i-1,j} - 4q_{ij}\Bigr).
\end{equation*}
Since the integration is performed with constant step size, for the experiment we took a slowly-changing exact solution with
\begin{equation}\label{eq:pq}
	p(t) = 1 + \cos(t),\quad q_{ij} = \exp(x_i + y_j)\sin(2 \pi x_i)\sin(3\pi y_j).
\end{equation}
The interval of integration is $[0, 10]$, and the initial condition is obviously $w_{ij}(0)=p(0) q_{ij}$. 

\subsection{Implementation details}

In the computational experiment we compare the performance of MRMS($k$, $k$) methods with their classical $k$-step BDF counterparts, $k=1,2,\ldots 5$. For the both methods problem \eqref{eq:heat-mol} is represented in the general form
\begin{equation}\label{eq:wAwb}
	w'(t)=A w(t)+b(t),
\end{equation}
where $A$ is a well-known banded matrix of the two-dimensional discrete Laplacian, $w$ and $b$ are vectors obtained from matrices $(w_{ij})$ and $(b_{ij})$ by stacking their columns atop one another. Matrix $A$ is stored in Compressed Sparse Column format.

The integration with both MRMS and BDF methods is performed with fixed time step, the starting values are taken from the known exact solution \eqref{eq:exact_sol}, \eqref{eq:pq}.

The code for the experiment can be downloaded from the GitHub repository \cite{github}. It is written in Python language (v.~3.5.2) and uses libraries numpy 1.16.4 and scipy 1.3.0. The test were made on a laptop with 1.8 GHz Intel Core i5 processor and 8 Gb of RAM. 

\paragraph{BDF methods}

Since matrix $A$ and the time step $\tau$ are constant, only one LU decomposition of matrix $\tau A - c_k I$ is enough for the whole integration interval. Hence the most time-consuming part of each step is the backsubstitution for the solution of the system
\begin{equation}\label{eq:bdf-backsub}
	(\tau A - c_k I) y_k = c_{k-1} y_{k-1}\ldots + c_{k-p} y_{k-p} - b(t_k).
\end{equation}
The sparse LU factorization is performed using the function scipy.sparse.linalg.splu which uses the SuperLU library described in \cite{SuperLU}.

\paragraph{MRMS methods}

Here the most important stage is the solution of linear least squares problem \eqref{eq:ls-Wg}. In our code it is performed using the function scipy.linalg.lstsq which is a wrapper around the corresponding LAPACK routines. We used the SVD-based method by specifying the option \textsf{lapack\_driver='gelsd'}. 

\subsection{Results}

The numerical tests are performed with grid sizes $N=20$, $N=400$ and $N=1000$.
The integration interval $[0,10]$ was split into $M_s$ equal steps, where $M_s=50\cdot 2^s$, $s=0,1,\ldots,5$, for $N=20$ and $N=400$. To make the computation time shorter for the largest problem with $N=1000$ the number of steps was ten times less: $M_s=5\cdot 2^s$.

At each run the absolute error in the maximum norm at the endpoint ${t=10}$ and the time of execution (in seconds) are measured. The resulting "time-error" diagrams are plotted at Figure~\ref{fig:heat2d}.  

\begin{figure}\label{fig:heat2d}
\begin{tabular}{ccc}
    \includegraphics[width=11cm]{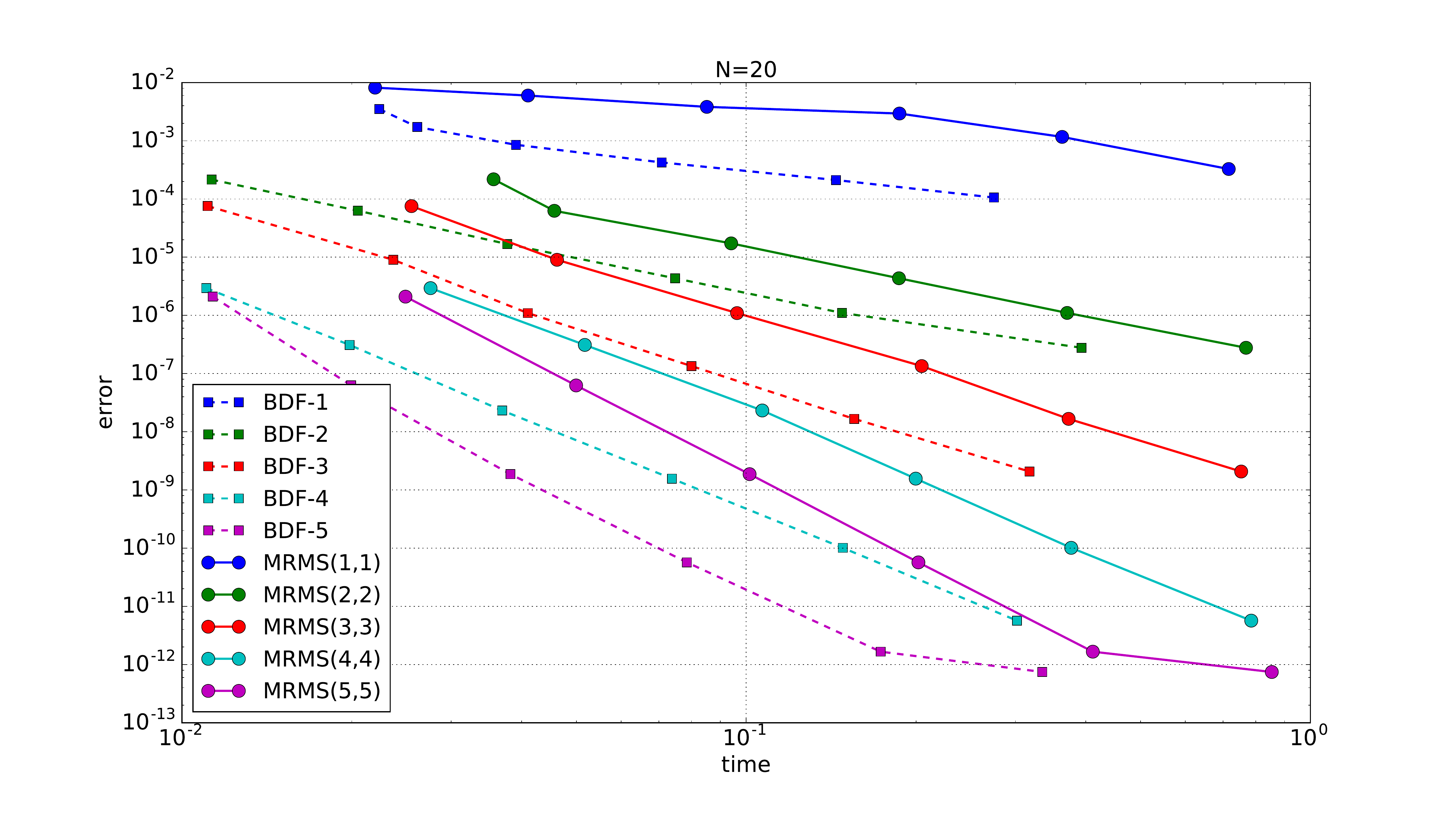} \\
    \includegraphics[width=11cm]{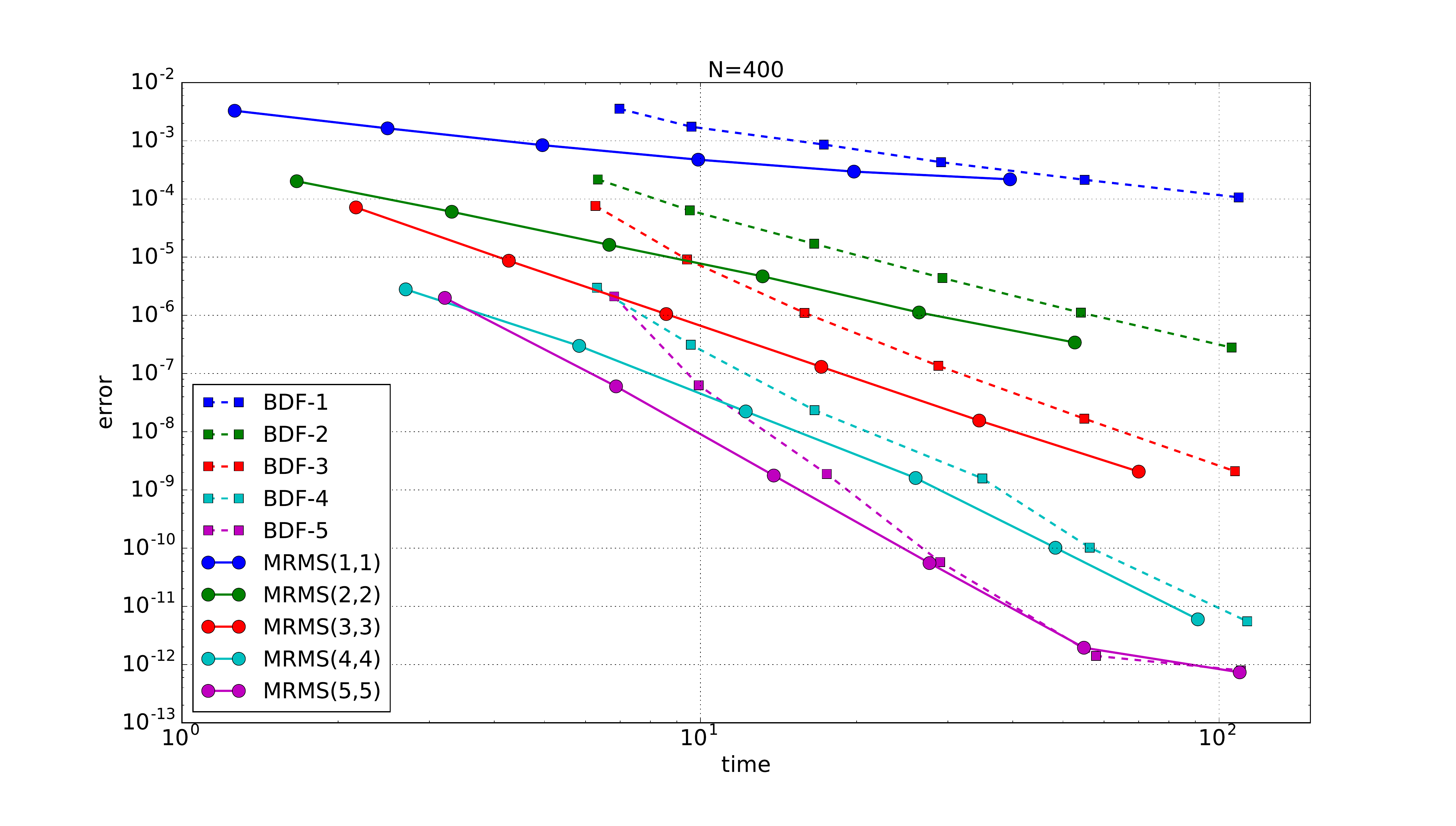} \\
    \includegraphics[width=11cm]{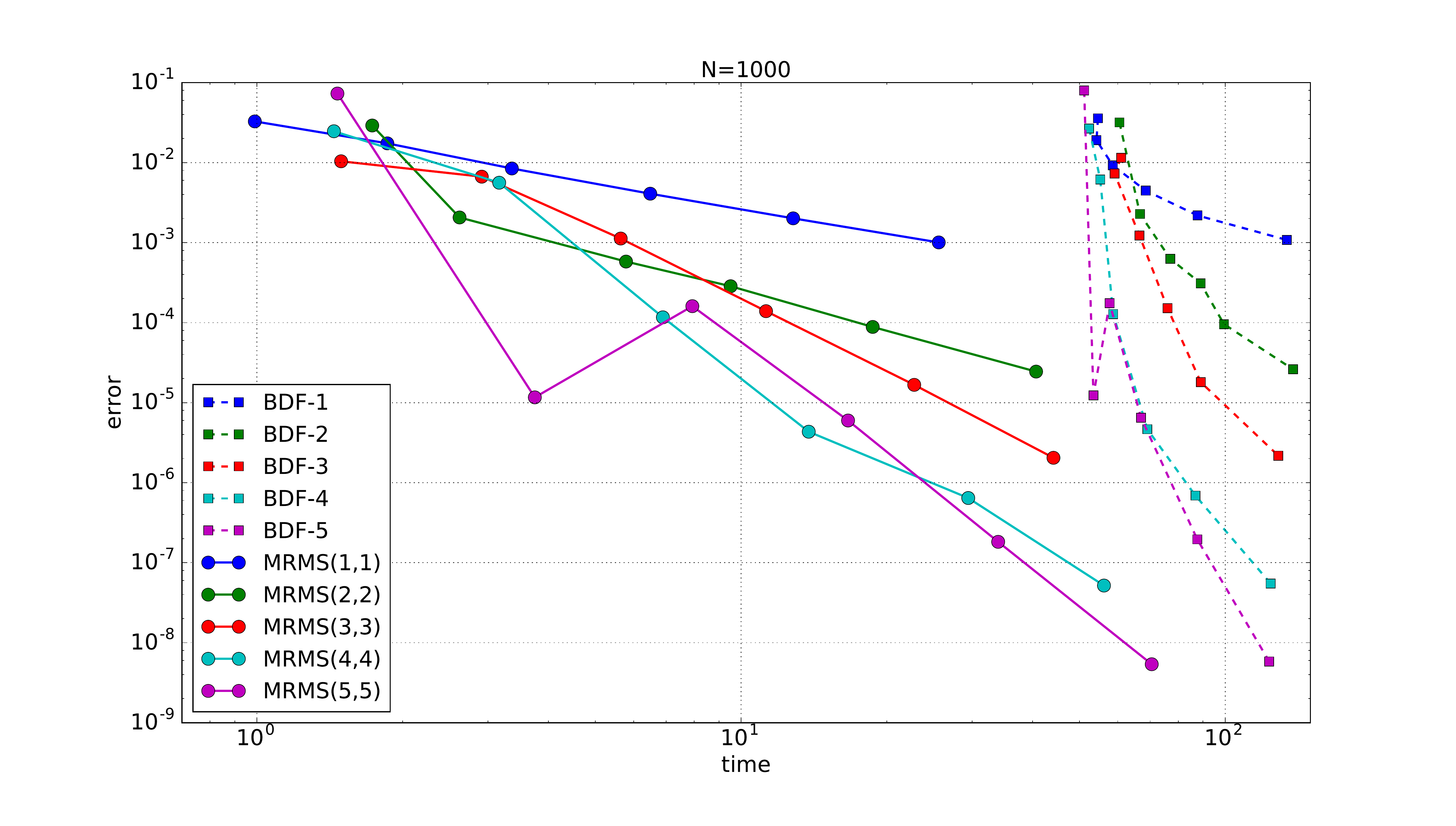} 
\end{tabular}
\caption{The results of numerical experiment with the 2D heat equation \eqref{eq:heat-mol}. The dimension of the ODE systems is $N^2$.}
\end{figure}

\subsection{Discussion}

The main observation regarding Figure \ref{fig:heat2d} is that the relative efficiency of the MRMS methods grows with the increase of dimension. For ${N=20}$ the BDF methods run faster, but already for ${N=400}$ the MRMS methods take a lead and for ${N=1000}$ they are several times faster. 

To give more evidence let's give some statistics for the case ${N=1000}$, ${k=5}$. \emph{The BDF method}: the initial computation of the LU decomposition took 52 seconds, one step took 0.5 seconds which include 0.47 seconds for the backsubstitution to solve \eqref{eq:bdf-backsub}. \emph{The MRMS method}: one step took 0.44 seconds including 0.3 seconds for the solution of the least-squares problem. Of course, all these numbers are platform- and implementation-specific, but the relation between the timings seems to be rather fair.  

Another remark is that on the problem being solved the general-purpose BDF solver with variable step size implemented in the function \textsf{scipy.integrate.solve\_ivp} (with option \textsf{method='BDF'}) for big values of $N$ performs significantly slower than our fixed-step implementation of the BDF method (for $N=400$ the difference is more than 10 times with the accuracy $10^{-6}$). One reason for this is that every time the step size is changed the LU decomposition must be recomputed, and another reason is that the general-purpose code took some additional time to compute the starting values. Needless to say that the explicit Runge--Kutta methods \textsf{RK45} and \textsf{RK23}, which are also available in \textsf{scipy.integrate.solve\_ivp}, are not competetive due to stiffness.

It is also worth mentioning that for $k\ne 1$ the errors of the BDF and MRMS methods are almost equal, which means that for this problem the exact solutions $\ybdf$ are well approximated by the elements of subspaces $\mcV$.

\section*{Conclusion}

In this paper we proposed a new approach to the numerical solution of big stiff linear systems. It is characterized by the reduced computation cost compared to usual implicit methods: instead of solving $n\times n$ linear systems on each step, the proposed numerical scheme requires solution of linear least squares problems with thin $n\times 2k$ matrices ($k \ll n$).

We showed that the MRMS methods inherit the accuracy and zero-stability properties from their BDF counterparts, and performed a partial analysis of linear stability for the one-step analog of implicit Euler method. This analysis showed that this method is applicable in stiff case. For the general case a numerical investigation of linear stability is performed. From the results of this investigation we can conclude that the MRMS methods are stable, but the accuracy in stiff case depends on the distribution of eigenvalues of matrix A: the difficulty of the problem grows as the number of eigenvalues which are close to zero increases. The accuracy can be improved by increasing the number of steps $k$ (whithout changing the order $p$).

In section \ref{sec:experiment} we compared the performance of the MRMS and BDF methods during the fixed step integration of a 2D heat equation on the unit square for different mesh sizes $N$. Though the BDF integration for the whole time interval employed only one sparse LU factorization procedure, for large $N$ even the cost of a backsubstitution exceeded the cost of the linear least squares problem solved on each step of the MRMS method. As a consequence the advantage of the MRMS methods in speed over the BDF methods was clearly seen for $N=400$ and became compelling for $N=1000$ (see Figure \ref{fig:heat2d}).

We would like to conclude with the list of three main topics which seem most important for further research on the MRMS methods:
\begin{enumerate}
    \item \emph{Convergence}. Though the approximation and zero-stability of the MRMS methods are proved in this paper, the rigorous proof of convergence is still missing. It is not clear if these well-known Dahlquist convergence conditions for the linear multistep methods \cite{Dahlquist} are sufficient in the MRMS case.

    \item \emph{Linear stability}. The standard scheme of linear stability investigation can not be directly applied to the MRMS methods, e.~g. it is impossible to construct a stability region in its usual sense for these methods. This is due to the intrinsic nonlinearity and multidimensionality of the methods. It is necessary to consider a set of eigenvalues at once and to reckon with the dependence of the evolutionary operator on the initial conditions. Therefore the rigorous linear stability and/or contractivity analysis of the MRMS methods is much-needed. 

    \item \emph{Nonlinear case}. From the technical point of view, the nonlinear case for the MRMS methods does not look as attractive as the linear one, because the numerical solution of the arising nonlinear optimization problem with standard techniques, such as gradient descent, will require expensive computations involving  Jacobian. Some workaround should be found for the MRMS approach to succeed on a nonlinear problem.

\end{enumerate}

\medskip 

\emph{The author would like to thank the anonymous reviewer for valuable comments and suggestions.}


\bibliographystyle{elsarticle-num}

\bibliography{faleichik-references}

\begin{thebibliography}{10}
\expandafter\ifx\csname url\endcsname\relax
  \def\url#1{\texttt{#1}}\fi
\expandafter\ifx\csname urlprefix\endcsname\relax\def\urlprefix{URL }\fi
\expandafter\ifx\csname href\endcsname\relax
  \def\href#1#2{#2} \def\path#1{#1}\fi

\bibitem{VDH}
P.~J.~Van Der~Houwen, B.~P.~Sommeuer, A special class of multistep
  runge—kutta methods with extended real stability interval, Ima Journal of
  Numerical Analysis - IMA J NUMER ANAL 2 (1982) 183--209.
\newblock \href {http://dx.doi.org/10.1093/imanum/2.2.183}
  {\path{doi:10.1093/imanum/2.2.183}}.

\bibitem{Lebedev}
V.~Lebedev, \href{http://dx.doi.org/10.1007/3-540-62598-4\_104}{Explicit
  difference schemes with variable time steps for solving stiff systems of
  equations}, in: L.~Vulkov, J.~Wasniewski, P.~Yalamov (Eds.), Numerical
  Analysis and Its Applications, Vol. 1196 of Lecture Notes in Computer
  Science, Springer Berlin Heidelberg, 1997, pp. 274--283.
\newblock \href {http://dx.doi.org/10.1007/3-540-62598-4\_104}
  {\path{doi:10.1007/3-540-62598-4\_104}}.
\newline\urlprefix\url{http://dx.doi.org/10.1007/3-540-62598-4\_104}

\bibitem{gpi}
B.~Faleichik, I.~Bondar, V.~Byl,
  \href{http://www.sciencedirect.com/science/article/pii/S037704271300589X}{Generalized
  picard iterations: A class of iterated runge–kutta methods for stiff
  problems}, Journal of Computational and Applied Mathematics 262 (2014) 37 --
  50, selected Papers from NUMDIFF-13.
\newblock \href {http://dx.doi.org/https://doi.org/10.1016/j.cam.2013.10.036}
  {\path{doi:https://doi.org/10.1016/j.cam.2013.10.036}}.
\newline\urlprefix\url{http://www.sciencedirect.com/science/article/pii/S037704271300589X}

\bibitem{Brown}
P.~N. Brown, A.~C. Hindmarsh,
  \href{http://dx.doi.org/10.1137/0723039}{Matrix-free methods for stiff
  systems of ode's}, SIAM J. Numer. Anal. 23~(3) (1986) 610--638.
\newblock \href {http://dx.doi.org/10.1137/0723039}
  {\path{doi:10.1137/0723039}}.
\newline\urlprefix\url{http://dx.doi.org/10.1137/0723039}

\bibitem{Rosenbrock}
H.~H. Rosenbrock, \href{https://dx.doi.org/10.1093/comjnl/5.4.329}{{Some
  general implicit processes for the numerical solution of differential
  equations}}, The Computer Journal 5~(4) (1963) 329--330.
\newblock \href {http://dx.doi.org/10.1093/comjnl/5.4.329}
  {\path{doi:10.1093/comjnl/5.4.329}}.
\newline\urlprefix\url{https://dx.doi.org/10.1093/comjnl/5.4.329}

\bibitem{Curtiss}
C.~Curtiss, J.~O. Hirschfelder, Integration of stiff equations, Proceedings of
  the National Academy of Sciences 38~(3) (1952) 235--243.

\bibitem{Ortega}
J.~M. Ortega, W.~C. Rheinboldt, Iterative solution of nonlinear equations in
  several variables., Computer science and applied mathematics, Academic Press,
  1970.

\bibitem{Saad}
Y.~Saad, Iterative Methods for Sparse Linear Systems, 2nd Edition, Society for
  Industrial and Applied Mathematics, Philadelphia, PA, USA, 2003.

\bibitem{hw-vol2}
E.~Hairer, G.~Wanner, Solving Ordinary Differential Equations {II}: Stiff and
  Differential-Algebraic Problems, 2nd Edition, Springer-Verlag, New York,
  1996.

\bibitem{demmel97}
J.~W. Demmel, Applied Numerical Linear Algebra, SIAM, 1997.
\newblock \href {http://dx.doi.org/10.1137/1.9781611971446}
  {\path{doi:10.1137/1.9781611971446}}.

\bibitem{github}
B.~Faleichik, Mrms methods experimental code,
  \url{https://github.com/bfaleichik/mrms} (2019).

\bibitem{SuperLU}
X.~S. Li, M.~Shao, \href{http://doi.acm.org/10.1145/1916461.1916467}{A
  supernodal approach to incomplete lu factorization with partial pivoting},
  ACM Trans. Math. Softw. 37~(4) (2011) 43:1--43:20.
\newblock \href {http://dx.doi.org/10.1145/1916461.1916467}
  {\path{doi:10.1145/1916461.1916467}}.
\newline\urlprefix\url{http://doi.acm.org/10.1145/1916461.1916467}

\bibitem{Dahlquist}
G.~Dahlquist, \href{http://www.jstor.org/stable/24490010}{Convergence and
  stability in the numerical integration of ordinary differential equations},
  Mathematica Scandinavica 4~(1) (1956) 33--53.
\newline\urlprefix\url{http://www.jstor.org/stable/24490010}

\end{thebibliography}

\end{document}